\documentclass[12pt]{amsart}
\usepackage{amsmath,amssymb,amsfonts,amsthm,amscd}
\usepackage[all]{xy}
\usepackage{lscape} 
\usepackage{fullpage} 
\usepackage{xcolor}
\usepackage{ulem}
\usepackage[hidelinks]{hyperref}

\newtheorem{theorem}{Theorem}[section]
\newtheorem{lemma}[theorem]{Lemma}
\newtheorem{corollary}[theorem]{Corollary}
\newtheorem{proposition}[theorem]{Proposition}
\newtheorem*{theorem*}{Theorem}

\theoremstyle{definition}

\newtheorem{example}[theorem]{Example}

\theoremstyle{remark}
\newtheorem{remark}[theorem]{Remark}

\DeclareMathOperator{\fpmod}{mod}
\DeclareMathOperator{\Mod}{Mod}
\DeclareMathOperator{\proj}{proj}
\DeclareMathOperator{\inj}{inj}
\DeclareMathOperator{\Proj}{Proj}

\DeclareMathOperator{\Hom}{Hom}
\DeclareMathOperator{\End}{End}

\begin{document}
\title{Lifting and restricting t-structures}
\author{Frederik Marks and Alexandra Zvonareva}
\address{Frederik Marks, Institut f\"ur Algebra und Zahlentheorie, Universit\"at Stuttgart, Pfaffenwaldring 57, 70569 Stuttgart, Germany}
\email{marks@mathematik.uni-stuttgart.de}
\address{Alexandra Zvonareva, Institut f\"ur Algebra und Zahlentheorie, Universit\"at Stuttgart, Pfaffenwaldring 57, 70569 Stuttgart, Germany}
\email{alexandra.zvonareva@mathematik.uni-stuttgart.de}
\thanks{The first named author was supported by a postdoctoral fellowship of the Baden-W\"urttemberg Stiftung. Both authors acknowledge support from DFG (Deutsche Forschungsgemeinschaft) through a scientific network on silting theory.}


\begin{abstract}
We explore the interplay between t-structures in the bounded derived category of finitely presented modules and the unbounded derived category of all modules over a coherent ring $A$ using homotopy colimits. More precisely, we show that every intermediate t-structure in $D^b(\fpmod(A))$ can be lifted to a compactly generated t-structure in $D(\Mod(A))$, by closing the aisle and the coaisle of the t-structure under directed homotopy colimits. Conversely, we provide necessary and sufficient conditions for a compactly generated t-structure in $D(\Mod(A))$ to restrict to an intermediate t-structure in $D^b(\fpmod(A))$, thus describing which t-structures can be obtained via lifting. We apply our results to the special case of HRS-t-structures. Finally, we discuss various applications to silting theory in the context of finite dimensional algebras.
\end{abstract}
\maketitle


\section{introduction}
A modern perspective on representation theory suggests to view the category of modules over a ring as an abelian subcategory of its derived category, and to study properties of the ring via complexes of modules. Such a triangulated point of view allows to compare the representation theories of two given rings even if their module categories are rather different. Foundational work in this direction goes back to Happel \cite{H}, Rickard \cite{R} and Keller \cite{K}, who extended classical Morita theory to the setup of derived categories, explaining when two derived module categories are triangle equivalent. 
The key notion in this context are tilting complexes. In recent years, tilting theory was generalized to silting theory, which allows to study more generally the embeddings of module categories into derived categories as hearts of associated t-structures (see e.g. \cite{AI,KY,NSZ,PV}).

In general, t-structures play a fundamental role in understanding the structure of triangulated categories such as derived module categories. 
Following \cite{BBD}, every t-structure gives rise to an abelian category -- called the heart of the t-structure -- connected to the ambient triangulated category via a cohomological functor. In other words, every t-structure induces a cohomology theory on the ambient triangulated category. Nowadays, t-structures are prominently used in 
 different branches of mathematics such as algebraic geometry, representation theory, category theory and topology. 
 
 In particular, understanding bounded t-structures on the bounded derived category of finitely presented modules $D^b(\fpmod(A))$ over a finite dimensional algebra $A$ or on the bounded derived category of coherent sheaves on a smooth projective variety $\mathbb{X}$ is crucial for computing the space of Bridgeland stability conditions. In representation theory, bounded t-structures appear in the fundamental correspondence with silting complexes, co-t-structures and simple minded collections (see \cite{KV,KY,RR}).
 Nevertheless, even for bounded derived categories of finite dimensional algebras, many questions are still open, in particular, when it comes to constructing t-structures explicitly.
On the contrary, in the unbounded derived category of all modules $D(\Mod(A))$ over an arbitrary ring $A$, one can always generate a t-structure from any given set of objects \cite{AJS,SS}.  
 
In this paper, we explore for a coherent ring $A$ the interplay between t-structures in  $D^b(\fpmod(A))$ and in $D(\Mod(A))$. We restrict ourselves to intermediate t-structures, sitting in an interval between shifts of the standard t-structure. Note that if $A$ is a finite dimensional algebra, a t-structure in $D^b(\fpmod(A))$ is intermediate if and only if it is bounded.

Building on \cite{SSV}, where it was proved that t-structures in $D(\Mod(A))$ lift to coherent diagrams, we show that every intermediate t-structure $(\mathcal{X},\mathcal{Y})$ in $D^b(\fpmod(A))$ with heart $\mathcal{H}$ can be lifted to a compactly generated t-structure $(\underrightarrow{\operatorname{hocolim}}(\mathcal{X}),\underrightarrow{\operatorname{hocolim}}(\mathcal{Y}))$ in $D(\Mod(A))$ by closing the classes $\mathcal{X}$ and $\mathcal{Y}$ under directed homotopy colimits.   Moreover, the heart of this new t-structure is a locally coherent Grothendieck category with the set of finitely presented objects given by $\mathcal{H}$ (see Section \ref{lifting}). 
This process of lifting t-structures from $D^b(\fpmod(A))$ to $D(\Mod(A))$ can be considered as a generalization of lifting torsion pairs $(\mathcal{T},\mathcal{F})$ from $\fpmod(A)$ to torsion pairs $(\varinjlim(\mathcal{T}),\varinjlim(\mathcal{F}))$ in $\Mod(A)$ (see \cite[Section 4.4]{CB}).

In Section \ref{restricting}, we discuss the problem of restricting t-structures from $D(\Mod(A))$ to $D^b(\fpmod(A))$. Note that except for the commutative case discussed in \cite{AJSa} not much seems to be known. We prove that a compactly generated and intermediate t-structure in $D(\Mod(A))$ restricts to $D^b(\fpmod(A))$ if and only if every object in the heart $\mathbf{H}$ is a directed homotopy co\-limit of objects in $\mathbf{H}\cap D^b(\fpmod(A))$ (see Theorem \ref{thm rest}). Altogether, via lifting and restricting t-structures, we obtain the following result.

\begin{theorem*}[see Corollary \ref{bijection}]
Lifting and restricting t-structures yields a bijetcion between
\begin{enumerate}
\item intermediate t-structures in $D^b(\fpmod(A))$;
\item intermediate, compactly generated t-structures in $D(\Mod(A))$ with a locally coherent Grothendieck heart $\mathbf{H}$, whose finitely presented objects are given by $\mathbf{H}\cap D^b(\fpmod(A))$.
\end{enumerate}
\end{theorem*}

As a consequence, we get that a derived equivalence $D(\Mod(B))\overset{F}{\longrightarrow} D(\Mod(A))$ preserves coherent rings if and only if the image under $F$ of the standard t-structure in $D(\Mod(B))$ restricts to $D^b(\fpmod(A))$ (see Corollary \ref{coherent} and Remark \ref{rem}). 
Moreover, when restricting the bijection above to the special case of noetherian rings and HRS-t-structures, that is intermediate t-structures with the smallest possible interval, the bijection takes a simpler form and recovers some results on lifting torsion pairs (see Section \ref{section HRS}).

In Section \ref{section silting}, we discuss applications to silting theory. First of all, it turns out that every lifted t-structure in $D(\Mod(A))$ is controlled by a not necessarily small cosilting complex. 
In case $A$ is a finite dimensional algebra we can establish a dual way of lifting by cogenerating a t-structure in $D(\Mod(A))$ from a t-structure in $D^b(\fpmod(A))$. In this way, we obtain silting t-structures with respect to not necessarily compact silting complexes (see Theorem \ref{thm silt}). In general, we do not know how to classify the silting t-structures of this form. Only in the case of HRS-t-structures we get a complete picture (see Proposition \ref{HRS dual}). Finally, by combining both ways of lifting, we can show that for an intermediate t-structure $(\mathcal{X},\mathcal{Y})$ in $D^b(\fpmod(A))$ there is a unique t-structure in $D(\Mod(A))$ restricting to $(\mathcal{X},\mathcal{Y})$ if and only if $(\mathcal{X},\mathcal{Y})$ is algebraic, i.e. its heart is a finite length category (see Corollary \ref{cor restrict}).

\bigskip

{\bf Acknowledgments.} The authors are grateful to Michal Hrbek, Simone Virili and Jorge Vit\'oria for helpful discussions on different parts of the presented material. The authors also thank the anonymous referee for their careful reading and valuable comments. 


\section{Preliminaries}
\subsection{Notation.}
Throughout, let $A$ be a \textbf{right coherent ring}, that is a ring $A$ such that every finitely generated submodule of a finitely presented right $A$-module is again finitely presented. The category of right $A$-modules is denoted by $\Mod(A)$, its subcategory of finitely presented modules by $\fpmod(A)$, its subcategory of projective modules by $\Proj(A)$ and its subcategory of finitely generated projective modules by $\proj(A)$. Note that $A$ is right coherent if and only if $\fpmod(A)$ is an abelian category. By $K^b(\proj(A))$ (respectively, $K^-(\proj(A))$) we denote the (right) bounded homotopy category. We write $D(A)=D(\Mod(A))$ for the  derived category of all $A$-modules and $D^b(A)=D^b(\fpmod(A))$ for the bounded derived category of finitely presented $A$-modules.

All subcategories considered are strict and full. If $\mathcal{C}$ is a subcategory of an additive category $\mathcal{A}$, we define 
$\mathcal{C}^\perp$ to be the subcategory of $\mathcal{A}$ formed by all objects $Z$ with $\Hom_\mathcal{A}(C,Z)=0$ for all $C$ in $\mathcal{C}$. 
Dually, we define $^\perp\mathcal{C}$. 
If $\mathcal{T}$ is a triangulated category and $C$ an object in $\mathcal{T}$, we define the following perpendicular classes for any subset $I$ of $\mathbb{Z}$:
$$^{\perp_I}C:=\{Z\in \mathcal{T}\mid \Hom_\mathcal{T}(Z,C[i])=0 \text{ for all } i\in I\}$$
$$C^{\perp_I}:=\{Z\in \mathcal{T}\mid \Hom_\mathcal{T}(C,Z[i])=0 \text{ for all } i\in I\}$$

\subsection{t-structures.}
For a triangulated category $\mathcal{T}$ and a subcategory $\mathcal{C}$ of $\mathcal{T}$, we call a morphism $f\colon Z\longrightarrow C$ in $\mathcal{T}$ a {\bf left $\mathcal{C}$-approximation} of $Z$ provided that $C$ is in $\mathcal{C}$ and the induced map $\Hom_\mathcal{T}(C,C')\longrightarrow\Hom_\mathcal{T}(Z,C')$ is surjective for all $C'$ in $\mathcal{C}$. We call $f$ a {\bf minimal left $\mathcal{C}$-approximation}, if $f$ is a {\bf left-minimal} morphism, i.e. every morphism $g\colon C\longrightarrow C$ with $g\circ f=f$ is an isomorphism. Dually, we define (minimal) right approximations.

A pair of subcategories $(\mathcal{X},\mathcal{Y})$ in $\mathcal{T}$ is said to be a {\bf torsion pair}, if both $\mathcal{X}$ and $\mathcal{Y}$ are closed under direct summands, $\Hom_\mathcal{T}(X,Y)=0$ for all $X$ in $\mathcal{X}$ and $Y$ in $\mathcal{Y}$, and for all $Z$ in $\mathcal{T}$ there is a triangle 
\begin{equation*}
\xymatrix{X\ar[r] & Z \ar[r] & Y \ar[r] & X[1]}\tag{$\star$}
\end{equation*}
with $X$ in $\mathcal{X}$ and $Y$ in $\mathcal{Y}$. Moreover, a torsion pair $(\mathcal{X},\mathcal{Y})$ is called a {\bf t-structure} (respectively, {\bf co-t-structure}), if $\mathcal{X}$ is closed under positive (respectively, negative) shifts. We call $\mathcal{X}$ the {\bf aisle} and $\mathcal{Y}$ the {\bf coaisle} of the (co-)t-structure. 

For a t-structure $(\mathcal{X},\mathcal{Y})$ the morphism $X\longrightarrow Z$ in the triangle above is a minimal right $\mathcal{X}$-approximation, and the morphism $Z\longrightarrow Y$ is a minimal left $\mathcal{Y}$-approximation. These approximations are induced by the {\bf truncation functors} $\tau_\mathcal{X}\colon \mathcal{T}\longrightarrow \mathcal{X}$, which is right adjoint to the inclusion of $\mathcal{X}$ in $\mathcal{T}$, and $\tau^\mathcal{Y}\colon \mathcal{T}\longrightarrow \mathcal{Y}$, which is left adjoint to the inclusion of $\mathcal{Y}$ in $\mathcal{T}$. The triangle $(\star)$ then takes the form
$$\tau_{\mathcal{X}}Z\xymatrix{\ar[r] & Z \ar[r] & \tau^{\mathcal{Y}}Z \ar[r] & \tau_\mathcal{X}Z[1]}.$$
Moreover, for a t-structure $(\mathcal{X},\mathcal{Y})$ we denote by $\mathcal{H}:=\mathcal{X}\cap\mathcal{Y}[1]$ its {\bf heart}, which is well-known to be an abelian category. Its exact structure is given by the triangles in $\mathcal{T}$ with the first three terms in $\mathcal{H}$.
Combining the truncation functors, we get a cohomological functor
$H^0\colon\mathcal{T}\longrightarrow\mathcal{H}$
mapping an object $Z\in \mathcal{T}$ to 
$\tau^{\mathcal{Y}[1]} \tau_\mathcal{X}(Z)$, where $\tau^{\mathcal{Y}[1]}$ denotes the truncation functor with respect to the shifted t-structure $(\mathcal{X}[1],\mathcal{Y}[1])$.
The functor $H^0$ sends triangles in $\mathcal{T}$ to long exact sequences in $\mathcal{H}$.

If $\mathcal{T}$ is a triangulated category with coproducts, we say that an object $Z$ in $\mathcal{T}$ is {\bf compact}, if the functor $\Hom_\mathcal{T}(Z,-)$ commutes with coproducts. Moreover, a t-structure $(\mathcal{X},\mathcal{Y})$ in $\mathcal{T}$ is called {\bf compactly generated}, if $\mathcal{Y}=\mathcal{S}^\perp$ for a set $\mathcal{S}$ of compact objects in $\mathcal{T}$.

For $\mathcal{T}=D(A)$, we denote by $(D(A)^{\leq 0},D(A)^{\geq 1})$ the {\bf standard t-structure}, where $D(A)^{\leq 0}$ (respectively, $D(A)^{\geq 1}$) denotes the subcategory of complexes with non-zero cohomologies only in non-positive (respectively, positive) degrees. For $n\in\mathbb{Z}$, we define $D(A)^{\leq n}:=D(A)^{\leq 0}[-n]$ and $D(A)^{\geq n+1}:=D(A)^{\geq 1}[-n]$.
For the standard t-structure $(D(A)^{\leq 0},D(A)^{\geq 1})$ we denote the truncation functors by $\tau^{\leq 0}$ and $\tau^{\geq 1}$, for the shift of  the standard t-structure $(D(A)^{\leq n},D(A)^{\geq n+1})$ by $\tau^{\leq n}$ and $\tau^{\geq n+1}$.
A t-structure $(\mathcal{X},\mathcal{Y})$ in $D(A)$ is called {\bf intermediate}, if there are $m,n\in\mathbb{Z}$ such that $D(A)^{\leq m}\subseteq\mathcal{X}\subseteq D(A)^{\leq n}$, or equivalently, $D(A)^{\geq n+1}\subseteq\mathcal{Y}\subseteq D(A)^{\geq m+1}$.
Note that $D^b(A)$ is equipped with a standard t-structure as well, and we can adapt the previous notation. In particular, we call a t-structure $(\mathcal{X},\mathcal{Y})$ in $D^b(A)$ {\bf intermediate}, if there are $m,n\in\mathbb{Z}$ such that $D^b(A)^{\leq m}\subseteq\mathcal{X}\subseteq D^b(A)^{\leq n}$. If $A$ is a finite dimensional algebra, it can be checked that a t-structure $(\mathcal{X},\mathcal{Y})$ in $D^b(A)$ is intermediate if and only if it is {\bf bounded}, i.e.
$$\bigcup\limits_{n\in\mathbb{Z}}\mathcal{X}[n]=D^b(A)=\bigcup\limits_{n\in\mathbb{Z}}\mathcal{Y}[n].$$
\vspace{-11pt}

\subsection{Homotopy colimits.}
We will be interested in taking homotopy colimits in $D(A)$. For that we will follow the approach based on the theory of derivators, without actually defining all the corresponding notions. Instead, we will give a short overview of facts and ideas used to define homotopy colimits in $D(A)$.  
For a more detailed overview on homotopy colimits in triangulated categories and the use of derivators, we refer to \cite{Groth,Hr,SSV,S}.

Let $Cat$ be the
2-category of all small categories. For $I\in Cat$ we consider the category $\Mod(A)^{I}$, the category of $I$-shaped diagrams of $A$-modules. Objects of this category are functors $I\longrightarrow \Mod(A)$ and morphisms are natural transformations between these functors. The category $\Mod(A)^{I}$ is again abelian, so we can consider the unbounded derived category $D(\Mod(A)^{I})$. This category is the localization of the category of complexes $C(\Mod(A)^{I})\simeq C(\Mod(A))^{I}$ with respect to quasi-isomorphisms. The objects of $D(\Mod(A)^{I})$ are called {\bf coherent diagrams} and we think of them as $I$-shaped diagrams of complexes of $A$-modules.

A functor $u\colon I\longrightarrow J$ in $Cat$
yields a functor $u^*\colon \Mod(A)^{J} \longrightarrow \Mod(A)^{I}$ via precomposition. 
As $u^*$ is exact, it further induces a triangle functor $u^*\colon D(\Mod(A)^{J}) \longrightarrow D(\Mod(A)^{I})$.
Now let $\ast$ be the category with one object and one morphism. The category 
$D(\Mod(A)^{\ast})$ identifies with the category $D(A)$. For any $I\in Cat$ there exists a unique functor $\pi_I\colon I\longrightarrow \ast$.
The constant diagram functor $\pi_I^*\colon D(A) \longrightarrow D(\Mod(A)^{I})$
maps an object $Z\in D(A)$ to the constant diagram of shape $I$ with entries $Z$ and identity maps everywhere. We can view $\pi_I^*$ as a functor between the corresponding categories of complexes. In this case, its left (respectively, right) adjoint is the usual colimit (respectively, limit) functor in
the category of complexes.  The {\bf homotopy colimit} is defined to be the left derived functor of this colimit functor, and thus the functor left adjoint to $\pi_I^*$ on the level of derived categories.
$$
\xymatrix{ D(A)  \ar[rr]^{\pi_I^*}&&  \ar@/_2pc/[ll]_{\operatorname{hocolim}} D(\Mod(A)^{I})}
$$

We will mostly be interested in the case, when $I$ is a directed category. In this situation, we speak of a {\bf directed homotopy colimit} and write $\underrightarrow{\operatorname{hocolim}}_{i\in I}Z_i$, where $(Z_i\mid i\in I)$ denotes a coherent diagram. Since direct limits in the category of complexes of $A$-modules are exact, directed homotopy colimits can be computed as usual direct limits of complexes.

For a small category $I$ and an object $i\in I$ we can consider the functor $\iota_i\colon\ast \longrightarrow I$ sending $\ast$ to $i$, and the corresponding functor $\iota_i^*\colon D(\Mod(A)^{I}) \longrightarrow D(A)$. Evaluating $\iota_i^*$ on a diagram of complexes of $A$-modules  $(Z_i\mid i\in I)$ yields $Z_i$. Moreover, each morphism $f\colon i\longrightarrow j$ in $I$ gives rise to a natural transformation $f\colon \iota_i \Longrightarrow \iota_j$ and thus to a natural transformation   $f^*\colon \iota_i^* \Longrightarrow \iota_j^*$ (via precomposition). Evaluating $f^*$ on $(Z_i\mid i\in I)$ yields the corresponding morphism $Z_i\longrightarrow Z_j$. We can extend this assignment to define the diagram functor $\operatorname{dia}_I\colon D(\Mod(A)^{I}) \longrightarrow D(A)^{I}$. The objects of $D(A)^{I}$ are called {\bf incoherent diagrams}. 

Finally, for a subcategory $\mathcal{C}$ of $D(A)$, we denote by $\underrightarrow{\operatorname{hocolim}}(\mathcal{C})$ the smallest subcategory of $D(A)$ containing $\mathcal{C}$ and closed under directed homotopy colimits of objects from $\mathcal{C}$, that is closed under directed homotopy colimits of diagrams $Z=(Z_i\mid i\in I)$ in $D(\Mod(A)^{I})$ with $\iota_i^*(Z)=Z_i\in\mathcal{C}$ for all $i\in I$, for all directed categories $I\in Cat$.

\section{Lifting t-structures via homotopy colimits}\label{lifting}

In the proof of the next lemma we will use brutal truncations. For a complex $X=(X^i,d_i)$ with components $X^i$ and a differential $d_i$ of degree $1$ we denote the {\bf brutal truncation at degree $m$} by $\sigma^{\geq m}(X)$ (respectively $\sigma^{< m}(X)$), where $\sigma^{\geq m}(X)=(\sigma^{\geq m}(X)^i, d'_i)$ is the complex with $\sigma^{\geq m}(X)^i=X^i$ for $i\geq m$ and $0$ otherwise, and $d'_i=d_i$ for $i\geq m$ and $0$ otherwise. The truncation $\sigma^{< m}(X)$ is defined to be the complex $(\sigma^{< m}(X)^i, d'_i)$ with $\sigma^{< m}(X)^i=X^i$ for $i< m$ and $0$ otherwise, and $d'_i=d_i$ for $i< m-1$ and $0$ otherwise.

\begin{lemma}\label{comp.gen.}
Let $(\mathcal{X},\mathcal{Y})$ be an intermediate t-structure in $D^b(A)$. Then 
$$(\mathbf{X},\mathbf{Y}):= ({}^\perp(\mathcal{X}^\perp),\mathcal{X}^\perp)$$ 
is a compactly generated intermediate t-structure in $D(A)$ and
$$(\mathbf{X}\cap D^b(A),\mathbf{Y}\cap D^b(A))=(\mathcal{X},\mathcal{Y}).$$
\end{lemma}

\begin{proof}
Since $\mathcal{X}$ is skeletally small, the pair $(\mathbf{X},\mathbf{Y})$ forms a t-structure in $D(A)$, where $\mathbf{X}$ is given by the smallest subcategory of $D(A)$ containing $\mathcal{X}$ and closed under extensions, coproducts and positive shifts (see \cite{AJS}). By construction, the t-structure $(\mathbf{X},\mathbf{Y})$ is again intermediate, i.e. there are $m,n\in\mathbb{Z}$ such that $D(A)^{\leq m}\subseteq\mathbf{X}\subseteq D(A)^{\leq n}$. This follows from the fact that taking the orthogonal in $D(A)$ yields $\left(D^b(A)^{\leq n}\right)^\perp=D(A)^{\geq  n+1}$. Now take $X$ in $\mathcal{X}$ and view $X$ as an object in $K^-(\proj(A))$ by taking its projective resolution. Using brutal truncation, we obtain a triangle
$$\xymatrix{\sigma^{< m}(X)[-1]\ar[r] & \sigma^{\geq m}(X) \ar[r] & X \ar[r] & \sigma^{< m}(X)},$$
where $\sigma^{\geq m}(X)$, viewed as an object in $D(A)$, is compact. Define $\mathcal{S}:=\mathcal{S}_1\cup\mathcal{S}_2$ to be the set of compact objects given by $\mathcal{S}_1:=\{A[t]\mid t\geq -m\}$ and $\mathcal{S}_2:=\{\sigma^{\geq m}(X)\mid X\in \mathcal{X}\}$. We claim that $\mathcal{S}^\perp=\mathbf{Y}$, showing that $(\mathbf{X},\mathbf{Y})$ is a compactly generated t-structure. 

First, note that $\mathbf{Y}\subseteq D(A)^{\geq m+1}=\mathcal{S}_1^\perp$. Hence, it suffices to observe that for an object $Z$ in $\mathcal{S}_1^\perp$, we have
$$Z\in\mathcal{S}_2^\perp\Leftrightarrow Z\in\mathbf{Y}=\mathcal{X}^\perp.$$
Applying the functor $\Hom_{D(A)}(-,Z)$ to the triangle above and using the fact that both $\Hom_{D(A)}(\sigma^{< m}(X)[-1],Z)$ and $\Hom_{D(A)}(\sigma^{< m}(X),Z)$ are zero, we get the claim.
\end{proof}

We call the t-structure obtained in Lemma \ref{comp.gen.} the {\bf lifted t-structure associated to $(\mathcal{X},\mathcal{Y})$}. We want to better understand the t-structures in $D(A)$ that we obtain via lifting. Recall that an object $X$ in a Grothendieck category $\mathbf{H}$ is said to be {\bf finitely presented} provided that $\Hom_\mathbf{H}(X,-)$ commutes with direct limits. Moreover, $\mathbf{H}$ is called {\bf locally coherent}, if it has a set of finitely presented generators and if the finitely presented objects form an exact abelian subcategory.

\begin{theorem}\label{thm lifting}
Let $(\mathcal{X},\mathcal{Y})$ be an intermediate t-structure in $D^b(A)$ with heart $\mathcal{H}$, and let $(\mathbf{X},\mathbf{Y})$ be the associated lifted t-structure in $D(A)$ with heart $\mathbf{H}$. Then the following holds. 

\begin{enumerate}
	\item $(\mathbf{X},\mathbf{Y})=(\underrightarrow{\operatorname{hocolim}}(\mathcal{X}),\underrightarrow{\operatorname{hocolim}}(\mathcal{Y}))$.
	\item $\mathbf{H}=\underrightarrow{\operatorname{hocolim}}(\mathcal{H})$, and $\mathbf{H}$ is a locally coherent Grothendieck category, whose finitely presented objects are given by $\mathcal{H}=\mathbf{H}\cap D^b(A)$. 
\end{enumerate}
\end{theorem}

\begin{proof}
$(1)$: Every aisle of a t-structure in $D(A)$ is closed under homotopy colimits by \cite[Proposition 4.2]{SSV}. Moreover, since $(\mathbf{X},\mathbf{Y})$ is compactly generated by Lemma \ref{comp.gen.}, the coaisle $\mathbf{Y}$ is closed under directed homotopy colimits (see \cite[Proposition 5.6]{SSV}). Indeed, if $S$ is a compact object in $D(A)$ and $(Z_i\mid i\in I)$ is a directed coherent diagram, we have the canonical isomorphism
$$\Hom_{D(A)}(S,\underrightarrow{\operatorname{hocolim}}_{i\in I}Z_i)\cong\varinjlim {}_{i\in I}\Hom_{D(A)}(S,Z_i).$$
It follows that $\underrightarrow{\operatorname{hocolim}}(\mathcal{X})\subseteq\mathbf{X}$ and $\underrightarrow{\operatorname{hocolim}}(\mathcal{Y})\subseteq\mathbf{Y}$. Now take an object $Z$ in $D(A)$ and consider the approximation triangle
$$\xymatrix{X\ar[r] & Z \ar[r] & Y \ar[r] & X[1]}$$
with $X$ in $\mathbf{X}$ and $Y$ in $\mathbf{Y}$. We will show that $X$ is in $\underrightarrow{\operatorname{hocolim}}(\mathcal{X})$ and $Y$ is in $\underrightarrow{\operatorname{hocolim}}(\mathcal{Y})$, proving that $(\mathbf{X},\mathbf{Y})=(\underrightarrow{\operatorname{hocolim}}(\mathcal{X}),\underrightarrow{\operatorname{hocolim}}(\mathcal{Y}))$. 
To this end, write $Z\cong\varinjlim_{i\in I}Z_i$ in $C(\Mod(A))$, where $(Z_i\mid i\in I)\in C(\Mod(A))^I$ is a direct system of bounded complexes of finitely presented modules (see \cite[Lemmas 4.1.1(ii) and 5.1.1]{GR} for the existence of such a diagram, and \cite[Section 4]{CH} for further context). We can view $(Z_i\mid i\in I)$ as an object in $D(\Mod(A)^I)$ and get $Z\cong\underrightarrow{\operatorname{hocolim}}_{i\in I}Z_i$. By \cite[Theorem A]{SSV}, there is a t-structure $(\mathbf{X}_I,\mathbf{Y}_I)$ in $D(\Mod(A)^I)$, where $\mathbf{X}_I$, respectively $\mathbf{Y}_I$, is given by all $I$-shaped diagrams of complexes in $\mathbf{X}$, respectively $\mathbf{Y}$. Consider the associated approximation triangle 
$$\xymatrix{(X_i\mid i\in I)\ar[r] & (Z_i\mid i\in I) \ar[r] & (Y_i\mid i\in I) \ar[r] & (X_i\mid i\in I)[1]}$$
in $D(\Mod(A)^I)$. Applying the functors $\iota_i^*$, we obtain triangles of the form
$$\xymatrix{X_i\ar[r] & Z_i \ar[r] & Y_i \ar[r] & X_i[1]}$$
in $D(A)$. Since all the $Z_i$ belong to $D^b(A)$, so do all the $X_i$ and $Y_i$. In fact, we have found the approximation triangles of $Z_i$ with respect to the t-structure $(\mathcal{X},\mathcal{Y})$ in $D^b(A)$. Taking the homotopy colimit of the triangle in $D(\Mod(A)^I)$ yields a triangle
$$\xymatrix{\underrightarrow{\operatorname{hocolim}}_{i\in I} X_i\ar[r] & \underrightarrow{\operatorname{hocolim}}_{i\in I} Z_i \ar[r] & \underrightarrow{\operatorname{hocolim}}_{i\in I} Y_i \ar[r] & \underrightarrow{\operatorname{hocolim}}_{i\in I} X_i[1]}$$
in $D(A)$ isomorphic to 
$$\xymatrix{X\ar[r] & Z \ar[r] & Y \ar[r] & X[1]}.$$
It follows that $X$ belongs to $\underrightarrow{\operatorname{hocolim}}(\mathcal{X})$ and $Y$ belongs to $\underrightarrow{\operatorname{hocolim}}(\mathcal{Y})$, as desired.
Observe that the previous arguments applied to an arbitrary directed coherent diagram imply that truncation functors of compactly generated t-structures always commute with directed homotopy colimits.

$(2)$: We know that $\underrightarrow{\operatorname{hocolim}}(\mathcal{H})\subseteq\mathbf{H}$ by $(1)$. Conversely, take an object $H$ in $\mathbf{H}$ and write again $H\cong\underrightarrow{\operatorname{hocolim}}_{i\in I}Z_i$ with $Z_i$ in $D^b(A)$ and $I$ directed. By $H^0\colon D(A)\longrightarrow \mathbf{H}$ we denote the cohomological functor associated to the t-structure $(\mathbf{X},\mathbf{Y})$. Note that on objects in $D^b(A)$ the functor $H^0$ coincides with the cohomological functor associated to $(\mathcal{X},\mathcal{Y})$. 
Since $H^0$ is the composition of truncation functors, it commutes with directed homotopy colimits following $(1)$. Hence, we get 
$$H=H^0(H)\cong H^0(\underrightarrow{\operatorname{hocolim}}_{i\in I}Z_i)\cong\underrightarrow{\operatorname{hocolim}}_{i\in I}H^0(Z_i)$$
with $H^0(Z_i)$ in $\mathcal{H}$, as desired. 

The heart $\mathbf{H}$ of $(\mathbf{X},\mathbf{Y})$ is a Grothendieck category by \cite[Corollary 4.10]{AMV3} (see also \cite{SSt,SSV} for more general results in this direction). 
By \cite[Theorem A and Corollary 5.8]{SSV} we can lift incoherent diagrams in $\mathbf{H}$ to coherent diagrams and compute direct limits in $\mathbf{H}$ as directed homotopy colimits in $D(A)$.
Thus, it only remains to show that the set of finitely presented objects in $\mathbf{H}$ is given by $\mathcal{H}=\mathbf{H}\cap D^b(A)$. 
This follows from \cite[Proposition 4.5]{Sa}. 
In order to apply \cite[Proposition 4.5]{Sa} we only need to check that for each direct system $(H_i\mid i\in I)$ in $\mathbf{H}$ and for each $k \in \mathbb{Z}$,
 the canonical map $\varinjlim_{i\in I}\Hom_{D(A)}(A[k], H_i) \longrightarrow \Hom_{D(A)}(A[k], \varinjlim_{i\in I}H_i)$
 is an isomorphism. After lifting $(H_i\mid i\in I)$ to a coherent diagram, this follows from the fact that $\Hom_{D(A)}(A[k], -)$ is naturally isomorphic to the $-k$-th cohomology functor on complexes, which preserves direct limits.
\end{proof}

We have seen so far that every intermediate t-structure $(\mathcal{X},\mathcal{Y})$ in $D^b(A)$ lifts to an intermediate and compactly generated t-structure $(\underrightarrow{\operatorname{hocolim}}(\mathcal{X}),\underrightarrow{\operatorname{hocolim}}(\mathcal{Y}))$ in $D(A)$, and that this lifting process is compatible with the heart.
Next, we would like to better understand which t-structures in $D(A)$ can be obtained in this way.

\section{Restricting t-structures}\label{restricting}
We say that a t-structure $(\mathbf{X},\mathbf{Y})$ in $D(A)$ {\bf restricts to $D^b(A)$}, if $(\mathbf{X}\cap D^b(A),\mathbf{Y}\cap D^b(A))$ is a t-structure in $D^b(A)$. Clearly, every t-structure in $D(A)$ obtained via lifting restricts to the initial t-structure in $D^b(A)$.

Moreover, we call a t-structure $(\mathbf{X},\mathbf{Y})$ in $D(A)$ {\bf homotopically smashing}, if the coaisle $\mathbf{Y}$ is closed under directed homotopy colimits. As discussed in the proof of Theorem \ref{thm lifting}, every compactly generated t-structure is homotopically smashing.

\begin{theorem}\label{thm rest}
Let $(\mathbf{X},\mathbf{Y})$ be an intermediate, homotopically smashing t-structure in $D(A)$ with heart $\mathbf{H}$. Then the following are equivalent.
\begin{enumerate}
\item $(\mathbf{X},\mathbf{Y})$ restricts to $D^b(A)$;
\item $\mathbf{Y}=\underrightarrow{\operatorname{hocolim}}(\mathbf{Y}\cap D^b(A))$;
\item $\mathbf{H}=\underrightarrow{\operatorname{hocolim}}(\mathbf{H}\cap D^b(A))$.
\end{enumerate}
In particular, if these conditions are fulfilled, we also have $\mathbf{X}=\underrightarrow{\operatorname{hocolim}}(\mathbf{X}\cap D^b(A))$.
\end{theorem}

\begin{proof}
$(1)\Rightarrow (2),(3)$: By lifting again, we obtain the t-structure
$$(\underrightarrow{\operatorname{hocolim}}(\mathbf{X}\cap D^b(A)),\underrightarrow{\operatorname{hocolim}}(\mathbf{Y}\cap D^b(A)))$$
in $D(A)$, which coincides with $(\mathbf{X},\mathbf{Y})$, as both $\mathbf{X}$ and $\mathbf{Y}$ are closed under directed homotopy colimits, and thus contain the respective classes $\underrightarrow{\operatorname{hocolim}}(\mathbf{X}\cap D^b(A))$ and $\underrightarrow{\operatorname{hocolim}}(\mathbf{Y}\cap D^b(A))$. Hence, we get $(2)$, and also $(3)$ by Theorem \ref{thm lifting}(2).

$(2)\Rightarrow (1)$: Now assume that $\mathbf{Y}=\underrightarrow{\operatorname{hocolim}}(\mathbf{Y}\cap D^b(A))$. We show that for an object $Z$ in $D^b(A)$ the approximation triangle 
$$\xymatrix{X\ar[r] & Z \ar[r]^f & Y \ar[r] & X[1]}$$
with respect to $(\mathbf{X},\mathbf{Y})$ lies in $D^b(A)$. To this end, we first consider the case when $Z$ is a compact object. By assumption, we can write $Y=\underrightarrow{\operatorname{hocolim}}_{i\in I}Y_i$ with $Y_i$ in $\mathbf{Y}\cap D^b(A)$ for all $i$. Since $Z$ is compact, we have
$$f\in\Hom_{D(A)}(Z,\underrightarrow{\operatorname{hocolim}}_{i\in I}Y_i)\cong\varinjlim {}_{i\in I}\Hom_{D(A)}(Z,Y_i).$$
Thus, the map $f$ factors through some $Y_i$. Since $f$ is a minimal left $\mathbf{Y}$-approximation, the object $Y$ identifies with a direct summand of $Y_i$. In particular, the whole approximation triangle above lies in $D^b(A)$, as desired.

For the general case, assume $Z\in D^b(A)$. We view $Z$ as an object in $K^-(\proj(A))$ and use truncations. First, note that there exists $i\in \mathbb{Z}$ big enough such that $\tau^{>m}\sigma^{\geq -i}(Z)\simeq \tau^{>m}Z$, where $m$ is given by the intermediate condition $D(A)^{\leq m} \subseteq \mathbf{X} \subseteq D(A)^{\leq n}$.  Since $\mathbf{Y}\subseteq D(A)^{>m}$ we get that  $\tau^{\mathbf{Y}}\tau^{>m}W\simeq \tau^{\mathbf{Y}}W$ for any $W\in D(A)$. Thus $$Y=\tau^{\mathbf{Y}}Z\simeq  \tau^{\mathbf{Y}}\tau^{>m}Z\simeq \tau^{\mathbf{Y}}\tau^{>m}\sigma^{\geq -i}(Z) \simeq \tau^{\mathbf{Y}}\sigma^{\geq -i}(Z) \in D^b(A)$$ as $\sigma^{\geq -i}(Z)$ is compact.

$(3)\Rightarrow (1)$: Finally, assume that $\mathbf{H}=\underrightarrow{\operatorname{hocolim}}(\mathbf{H}\cap D^b(A))$. Take $Z$ in $D^b(A)$ and consider the approximation triangle
\begin{equation*}
\xymatrix{X\ar[r] & Z \ar[r] & Y \ar[r] & X[1]}\tag{$\dagger$}   
\end{equation*}
with respect to $(\mathbf{X},\mathbf{Y})$. Since $(\mathbf{X},\mathbf{Y})$ is intermediate, there is some $k\geq 0$  with $Z[k]$ in $\mathbf{X}$. Instead of the triangle $(\dagger)$ we consider the shifted triangle
$$\xymatrix{X[k]\ar[r] & Z[k] \ar[r] & Y[k] \ar[r] & X[k+1]},$$
which is the approximation triangle of $Z[k]$ with respect to the t-structure $(\mathbf{X}[k],\mathbf{Y}[k])$.
Hence, it is enough to show that for any object $Z$ in $\mathbf{X}\cap D^b(A)$ and $k\geq 0$, the approximation triangle
$$\xymatrix{X[k]\ar[r]^f & Z \ar[r] & Y[k] \ar[r] & X[k+1]}$$
with respect to $(\mathbf{X}[k],\mathbf{Y}[k])$ belongs to $D^b(A)$. We proceed by induction on $k$. The case $k=0$ is trivial.
For the case $k=1$, we argue in a similar way as in the proof of $(2)\Rightarrow (1)$. First, consider a compact object $Z$ in $\mathbf{X}\cap D^b(A)$ together with the approximation triangle
$$\xymatrix{X[1]\ar[r] & Z \ar[r]^{\hspace{-0.2cm}\phi} & Y[1] \ar[r] & X[2]}.$$
We have $Y[1]=H^0(Z)$, where $H^0$ denotes the cohomological functor associated to $(\mathbf{X},\mathbf{Y})$. By assumption, it follows that $H^0(Z)=\underrightarrow{\operatorname{hocolim}}_{i\in I}H_i$ with $H_i$ in $\mathbf{H}\cap D^b(A)$ for all $i$. Since $Z$ is compact, the map $\phi$ factors through some  $H_i$ and, thus, the previous triangle lies in $D^b(A)$. 
Now suppose that $Z$ is an arbitrary object in $\mathbf{X}\cap D^b(A)$. We view $Z$ as an object in $K^-(\proj(A))$ so that the brutal truncations $Z_i:=\sigma^{\geq -i}(Z)$ are compact for all $i$. For every $i$ we get a triangle
$$\xymatrix{\sigma^{<- i}(Z)[-1]\ar[r] & Z_i \ar[r] & Z \ar[r] & \sigma^{< -i}(Z)}.$$
Since $(\mathbf{X},\mathbf{Y})$ is intermediate, the object $\sigma^{< -i}(Z)[-1]$ will belong to $\mathbf{X}$ for sufficiently large $i$. In that case, $Z_i$ will also be in $\mathbf{X}$, as an extension of objects in $\mathbf{X}$. Moreover, we can choose $i$ big enough so that $\tau^{>m-1}Z\simeq \tau^{>m-1}\sigma^{\geq -i}(Z)=\tau^{>m-1}Z_i$, where $m$ is given by the intermediate condition $D(A)^{\leq m} \subseteq \mathbf{X} \subseteq D(A)^{\leq n}$. Hence, as in the proof of $(2)\Rightarrow (1)$ we get
$$Y[1]=\tau^{\mathbf{Y}[1]}Z\simeq  \tau^{\mathbf{Y}[1]}\tau^{>m-1}Z\simeq \tau^{\mathbf{Y}[1]}\tau^{>m-1}Z_i \simeq \tau^{\mathbf{Y}[1]}Z_i \in D^b(A)$$ as $Z_i$ is compact. 

For the induction step, let
$$\xymatrix{X'[k-1]\ar[r]^{\hspace{0.8cm}g} & Z \ar[r] & Y'[k-1] \ar[r] & X'[k]}$$
be the approximation triangle of $Z$ with respect to the t-structure $(\mathbf{X}[k-1],\mathbf{Y}[k-1])$, which belongs to $D^b(A)$ by the induction hypothesis. Moreover, let 
$$\xymatrix{X^{''}[k]\ar[r]^{\hspace{-0.4cm}h} & X'[k-1] \ar[r] & Y^{''}[k] \ar[r] & X^{''}[k+1]}$$
be the approximation triangle of $X'[k-1]$ with respect to $(\mathbf{X}[k],\mathbf{Y}[k])$. Note that, as $X'[k-1]$ belongs to $\mathbf{X}[k-1]\cap D^b(A)$, this triangle belongs to $D^b(A)$ by the base of induction ($k=1$). We obtain the following commutative diagram of morphisms
$$\xymatrix{X^{''}[k]\ar@{.>}[d]^c\ar[r]^{\hspace{-0.4cm}h} & X'[k-1]\ar[d]^g & &\\X[k]\ar[r]^f & Z \ar[r] & Y[k] \ar[r] & X[k+1]}$$
Note that the composition $gh$ factors through $f$, as $f$ is a right $\mathbf{X}[k]$-approximation. On the other hand, since $X[k]$ is in $\mathbf{X}[k]\subseteq\mathbf{X}[k-1]$, the map $f$ factors through $g$, as $g$ is a right $\mathbf{X}[k-1]$-approximation, and thus through the composition $gh$, as $h$ is a right $\mathbf{X}[k]$-approximation. Hence, we get a map $d\colon X[k]\longrightarrow X^{''}[k]$ such that $f=ghd=fcd$. The minimality of $f$ guarantees that $cd$ is an isomorphism, showing that $X[k]$ lies in $D^b(A)$, as a direct summand of $X^{''}[k]$. This finishes the proof.
\end{proof}

As an immediate consequence, it follows that every intermediate, homotopically smashing t-structure in $D(A)$ that restricts to $D^b(A)$ is necessarily compactly generated. We can now determine precisely, which t-structures we obtain from lifting. In fact, it turns out that the necessary conditions discussed in Lemma \ref{comp.gen.} and Theorem \ref{thm lifting} are also sufficient.

\begin{corollary}\label{bijection}
Lifting and restricting t-structures yields a bijection between
\begin{enumerate}
	\item intermediate t-structures $(\mathcal{X},\mathcal{Y})$ in $D^b(A)$;
	\item intermediate, compactly generated t-structures $(\mathbf{X},\mathbf{Y})$ in $D(A)$ with a locally coherent Grothendieck heart $\mathbf{H}$, whose finitely presented objects are given by $\mathbf{H}\cap D^b(A)$.
\end{enumerate}
Moreover, the t-structures in $(2)$ are precisely the intermediate homotopically smashing t-structures in $D(A)$ that restrict to $D^b(A)$.
\end{corollary}

\begin{proof}
By Lemma \ref{comp.gen.} and Theorem \ref{thm lifting}, lifting intermediate t-structures in $D^b(A)$ yields a well-defined, injective map from $(1)$ to $(2)$. Now take a t-structure $(\mathbf{X},\mathbf{Y})$ in $D(A)$ as in $(2)$. Using \cite[Theorem A and Corollary 5.8]{SSV}, we can compute direct limits in $\mathbf{H}$ as directed homotopy colimits in $D(A)$. Hence, it follows that $\mathbf{H}=\underrightarrow{\operatorname{hocolim}}(\mathbf{H}\cap D^b(A))$. By Theorem~\ref{thm rest}, the t-structure $(\mathbf{X},\mathbf{Y})$ restricts to $D^b(A)$, and we can recover it via lifting, as $(\mathbf{X},\mathbf{Y})$ is compactly generated. The final statement follows again from Theorem~\ref{thm rest}.
\end{proof}

We finish this section with an application of Theorems \ref{thm lifting} and \ref{thm rest} providing a criterion for a derived equivalence to preserve coherent rings. For that we need some further definitions.

A complex $T\in K^b(\proj(A))$ is called {\bf tilting}, if $\Hom_{D(A)}(T,T[i])=0$ for any $i\neq 0$ and if the smallest triangulated subcategory of $K^b(\proj(A))$ containing $T$ and closed under direct summands is $K^b(\proj(A))$ itself. We say that two rings $A$ and $B$ are {\bf derived equivalent} if the derived categories $D(A)$ and $D(B)$ are triangle equivalent. By \cite{R,K}, this happens if and only if there exists a tilting complex $T$ in $K^b(\proj(A))$ such that $\End_{D(A)}(T)\simeq B$. Under the triangle equivalence defined by the tilting complex $T$, the standard t-structure on $D(B)$ is mapped to the t-structure on $D(A)$ generated by $T$, that is the t-structure $(\mathbf{X},\mathbf{Y})=(T^{\perp_{> 0}},T^{\perp_{\leq 0}})$. Its heart $\mathbf{H}$ is equivalent to $\Mod(B)$.

\begin{corollary}\label{coherent}
Let $A$ be a right coherent ring and $T$ be a tilting complex in $K^b(\proj(A))$. Then the ring $B=\End_{D(A)}(T)$ is right coherent if and only if the t-structure $(T^{\perp_{> 0}},T^{\perp_{\leq 0}})$ restricts to $D^b(A)$.
\end{corollary}

\begin{proof}
If the t-structure $(T^{\perp_{> 0}},T^{\perp_{\leq 0}})$ restricts to $D^b(A)$, we can recover it via lifting. Thus, the finitely presented objects in its heart $\mathbf{H}$ are of the form $\mathbf{H}\cap D^b(A)$ and they form an abelian category. Since $\mathbf{H}$ is equivalent to $\Mod(B)$, this implies that the finitely presented objects in $\Mod(B)$ form an abelian category. So $B$ is right coherent.

If $B$ is right coherent the standard t-structure on $D(B)$ restricts to the standard t-structure on $D^b(B)$, and the triangle equivalence $D(B)\cong D(A)$ induced by $T$ restricts to the level of bounded derived categories (see \cite[Proposition 8.1]{R}). Since the standard t-structure on $D(B)$ identifies with $(T^{\perp_{> 0}},T^{\perp_{\leq 0}})$ under this equivalence, we get that $(T^{\perp_{> 0}},T^{\perp_{\leq 0}})$ restricts to $D^b(A)$, as desired.
\end{proof}

Note that Corollary \ref{coherent} also appears as \cite[Proposition 6.6]{PSV}.

\begin{remark}\label{rem}
In case we have a triangle equivalence $F\colon D(B)\longrightarrow D(A)$ that is hypothetically not given by a tilting complex $T$, we can still argue as above by considering the t-structure $(F(D(B)^{\leq 0}), F(D(B)^{\geq 1}))$ instead of $(T^{\perp_{> 0}},T^{\perp_{\leq 0}})$ in $D(A)$. In particular, the ring $B$ is right coherent if and only if this t-structure restricts to $D^b(A)$.
\end{remark}

\section{Case study: HRS-t-structures}\label{section HRS}

In this section, we assume that $A$ is a right noetherian ring and we specialize Corollary~\ref{bijection} to the setting of HRS-t-structures arising from torsion pairs in module categories. Recall that a torsion pair $(\mathcal{T},\mathcal{F})$ in $\fpmod(A)$  induces a t-structure $(\mathcal{X},\mathcal{Y})$ in $D^b(A)$  as follows:
$$\mathcal{X}=\{X\in D^b(A)\mid H^i(X)=0 \text{ for }i>0 \text{ and } H^0(X)\in \mathcal{T} \} $$
$$\mathcal{Y}=\{Y\in D^b(A)\mid H^i(Y)=0 \text{ for }i<0 \text{ and } H^0(Y)\in \mathcal{F} \},$$
where $H^i$ denotes the cohomological functor with respect to the shifted standard t-structure $(D^b(A)^{\leq i},D^b(A)^{\geq i+1})$. Analogously, a torsion pair $(\mathbf{T},\mathbf{F})$ in $\Mod(A)$ induces a t-structure in $D(A)$. We call these t-structures {\bf HRS-t-structures}; see \cite{HRS} for details. By \cite[Lemma 1.1.2]{Po}, HRS-t-structures are precisely the t-structures $(\mathcal{X},\mathcal{Y})$ in $D^b(A)$ (respectively, $D(A)$) with $D^b(A)^{\leq -1}\subseteq\mathcal{X}\subseteq D^b(A)^{\leq 0}$ (respectively, $D(A)^{\leq -1}\subseteq\mathcal{X}\subseteq D(A)^{\leq 0}$), i.e. intermediate t-structures with the smallest non-trivial interval $[m,n]$ (see also \cite{BR}).

Starting with a torsion pair $(\mathcal{T},\mathcal{F})$ in $\fpmod(A)$, considering the associated HRS-t-structure in $D^b(A)$ and lifting this t-structure, we obtain an HRS-t-structure in $D(A)$ corresponding to some torsion pair $(\mathbf{T},\mathbf{F})$ in $\Mod(A)$. Since $H^0$ sends directed homotopy colimits to direct limits in $\Mod(A)$, we get that 
$$(\mathbf{T},\mathbf{F})=(\varinjlim(\mathcal{T}),\varinjlim(\mathcal{F})).$$
This recovers a well-known construction of lifting torsion pairs from $\fpmod(A)$ to $\Mod(A)$ by closing $\mathcal{T}$ and $\mathcal{F}$ under direct limits of objects in $\mathcal{T}$ and $\mathcal{F}$, respectively (see \cite[Section 4.4]{CB}). Recall that a torsion pair $(\mathbf{T},\mathbf{F})$ in $\Mod(A)$ is said to be of {\bf finite type}, if the class $\mathcal{F}$ is closed under direct limits. 
Torsion pairs of the form $(\varinjlim(\mathcal{T}),\varinjlim(\mathcal{F}))$ provide examples of torsion pairs of finite type.
The following proposition makes precise the connection between lifting HRS-t-structures on the level of torsion pairs and on the level of t-structures.

\begin{proposition}\label{HRS}
Let $A$ be a right noetherian ring. Then
there is a commutative square of bijections as follows.
$$\xymatrix{{\left\{\begin{array}{c}\text{t-structures $(\mathcal{X},\mathcal{Y})$ in $D^b(A)$ with} \\ \text{$D^b(A)^{\leq -1}\subseteq\mathcal{X}\subseteq D^b(A)^{\leq 0}$} \end{array}\right\}}\ar[rr]^{\hspace{0.9cm}1:1}\ar@<-1ex>[dd]_{\underrightarrow{\operatorname{hocolim}}} & & {\left\{\begin{array}{c}\text{torsion pairs in $\fpmod(A)$} \end{array}\right\}}\ar[ll]\ar@<-1ex>[dd]_{\varinjlim} \\  \\{\left\{\begin{array}{c}\text{compactly generated t-structures }\\ \text{$(\mathbf{X},\mathbf{Y})$ in $D(A)$ with} \\ \text{$D(A)^{\leq -1}\subseteq\mathbf{X}\subseteq D(A)^{\leq 0}$}
\end{array}\right\}}\ar[rr]^{\hspace{0.9cm}1:1}\ar@<-1ex>[uu]_{_{}\operatorname{restriction}} & & {\left\{\begin{array}{c}\text{torsion pairs in $\Mod(A)$}\\ \text{of finite type}
\end{array}\right\}}\ar[ll]\ar@<-1ex>[uu]_{_{}\operatorname{restriction}} }$$
Moreover, a t-structure $(\mathbf{X},\mathbf{Y})$ in $D(A)$ with $D(A)^{\leq -1}\subseteq\mathbf{X}\subseteq D(A)^{\leq 0}$ is compactly generated if and only if its heart $\mathbf{H}$ is a Grothendieck category. In this case, $\mathbf{H}$ is locally coherent and the finitely presented objects of $\mathbf{H}$ are given by $\mathbf{H}\cap D^b(A)$.
\end{proposition}

\begin{proof}
The upper horizontal bijection was already mentioned above.
Since $A$ is right noetherian, every torsion pair in $\Mod(A)$ restricts to a torsion pair in $\fpmod(A)$. Thus, it is not hard to check that the vertical maps on the right hand side of the diagram induce mutually inverse bijections.
By \cite[Theorem 3.10]{BP}, a torsion pair in $\Mod(A)$ is of finite type if and only if the associated HRS-t-structure is compactly generated if and only if the heart of this t-structure is a Grothendieck category. In particular, we get the lower horizontal bijection. Now it remains to observe that lifting t-structures on the left hand side yields a commutative diagram, so lifting induces a bijection with restriction being its inverse. Note that this bijection is a special case of the bijection appearing in Corollary \ref{bijection}.
\end{proof}

We have just seen that every compactly generated t-structure $(\mathbf{X},\mathbf{Y})$ in $D(A)$ fulfilling $D(A)^{\leq -1}\subseteq\mathbf{X}\subseteq D(A)^{\leq 0}$ restricts to $D^b(A)$. The following example shows that a similar statement fails, if we allow the interval to be any larger.

\begin{example}\label{Kronecker}
Let $A$ be the Kronecker algebra over an algebraically closed field $k$. Let $(\mathbf{X},\mathbf{Y})$ be the t-structure in $D(A)$ generated by the stalk complexes $A[2]$ and $S_\lambda$, where $S_\lambda$ denotes a simple regular $A$-module for $\lambda$ in $k\cup\{\infty\}$. Thus, we have $\mathbf{Y}=\mathcal{S}^\perp$ for $\mathcal{S}=\{A[n+2]\oplus S_\lambda[n]\mid n\geq0\}$. It follows that $(\mathbf{X},\mathbf{Y})$ is compactly generated with $D(A)^{\leq -2}\subseteq\mathbf{X}\subseteq D(A)^{\leq 0}$. But $(\mathbf{X},\mathbf{Y})$ does not restrict to $D^b(A)$ by \cite[Corollary 4.5(1)]{ST}.
\end{example}

\section{Connections with silting theory}\label{section silting}

An object $T$ in $D(A)$ is called {\bf silting}, if the pair $(T^{\perp_{> 0}},T^{\perp_{\leq 0}})$ forms a t-structure in $D(A)$. Dually, an object $C$ in $D(A)$ is called {\bf cosilting}, if the pair $(^{\perp_{\leq 0}}C,^{\perp_{> 0}}\!C)$ forms a t-structure in $D(A)$. We refer to these t-structures as {\bf silting} and {\bf cosilting t-structures}; further details can be found in \cite{NSZ,PV}.

Note that it follows from \cite[Theorem 4.6]{L} that every intermediate and homotopically smashing t-structure $(\mathbf{X},\mathbf{Y})$ in $D(A)$ is cosilting with respect to a pure-injective cosilting object $C$ in $D(A)$ with bounded cohomology. In fact, $C$ has non-zero cohomology only in the interval determined by the intermediate condition, and we can choose $C$ to be a bounded complex of injective $A$-modules. In particular, every t-structure in $D(A)$ obtained from lifting via homotopy colimits is controlled by such a cosilting complex. In general, however, we do not know how to deduce from the properties of the cosilting complex, if the associated t-structure restricts to $D^b(A)$; only partial results are available. For example, cosilting t-structures with respect to 2-term cosilting complexes of injective $A$-modules are (up to shift) precisely HRS-t-structures with respect to torsion pairs of finite type in $\Mod(A)$ (see \cite[Corollary 3.9]{A}). Thus, if $A$ is right noetherian, they restrict to $D^b(A)$ by Proposition \ref{HRS}. On the other hand, it follows from Example \ref{Kronecker} that not every cosilting t-structure with respect to a 3-term cosilting complex restricts to $D^b(A)$.

Next, we discuss various conditions for the lifted t-structures to be silting as well. For this purpose, we restrict ourselves to the case when $A$ is a finite dimensional algebra, which allows us to use the fundamental correspondences from \cite{KY}.

\begin{proposition}\label{silting}
Let $A$ be a finite dimensional algebra over a field $k$. Let $(\mathcal{X},\mathcal{Y})$ be an intermediate t-structure in $D^b(A)$ with heart $\mathcal{H}$, and let $(\mathbf{X},\mathbf{Y})$ be the associated lifted t-structure in $D(A)$, whose heart is denoted by $\mathbf{H}$. Then the following conditions are equivalent.
\begin{enumerate}
\item $(\mathbf{X},\mathbf{Y})$ is a silting t-structure with respect to a silting object $T$;
\item $\mathcal{X}\cap K^b(\proj(A))$ is the coaisle of a co-t-structure in $K^b(\proj(A))$;
\item $\mathcal{H}\cong\fpmod(\Lambda)$ for a finite dimensional $k$-algebra $\Lambda$;
\item $\mathbf{H}\cong\Mod(\Lambda')$ for a finite dimensional $k$-algebra $\Lambda'$;
\item $(\mathbf{X},\mathbf{Y})$ is a silting t-structure with respect to a compact silting object $T'$.
\end{enumerate}
\end{proposition}

\begin{proof}
$(1)\Rightarrow(2)$: To show $(2)$ it suffices to check that every object $Z$ in $K^b(\proj(A))$ admits a left $\mathcal{X}$-approximation $f\colon Z\longrightarrow X$, where $X$ lies in $K^b(\proj(A))$. Indeed, since by assumption $K^b(\proj(A))$ is a Krull-Schmidt category, it would follow from \cite[Corollary 1.4]{KS} that we can choose $f$ left-minimal, which by the dual of \cite[Lemma 2.1]{J} implies that there is a triangle
$$\xymatrix{C\ar[r] & Z \ar[r]^f & X \ar[r] & C[1]}$$
in $K^b(\proj(A))$ with $C$ in $^\perp\mathcal{X}$, as desired. Now take $Z$ in $K^b(\proj(A))$. Since $(\mathbf{X},\mathbf{Y})$ is a silting t-structure, there is a co-t-structure $(^\perp\mathbf{X},\mathbf{X})$ in $D(A)$ and we get a left $\mathbf{X}$-approximation $f'\colon Z\longrightarrow X'$ (see \cite[Remark 4]{NSZ}). By assumption, we can write $X'=\underrightarrow{\operatorname{hocolim}}_{i\in I}X_i$ with $X_i$ in $\mathcal{X}$.
Since $Z$ is compact, the map $f'$ factors through some $X_i$ via a map $f_i'\colon Z\longrightarrow X_i$, which by construction is a left $\mathcal{X}$-approximation. Now we can view $X_i$ as an object in $K^-(\proj(A))$ and use brutal truncations to write $X_i\cong\underrightarrow{\operatorname{hocolim}}_{j\in\mathbb{N}}\sigma^{\geq -j}(X_i)$ with $\sigma^{\geq -j}(X_i)$ in $K^b(\proj(A))$. For sufficiently large $j$ the objects $\sigma^{\geq -j}(X_i)$ belong to $\mathcal{X}$ (see the proof of $(3)\Rightarrow (1)$ in Theorem \ref{thm rest}). Also, using again the fact that $Z$ is compact, we get that $f_i'$ factors through $\sigma^{\geq -j}(X_i)$ for $j$ big enough. This way, we get the required approximation.

{\it Remaining implications}:
Since $(\mathcal{X},\mathcal{Y})$ is intermediate, the co-t-structure in $(2)$ with coaisle $\mathcal{X}\cap K^b(\proj(A))$ is intermediate with respect to the standard co-t-structure in $K^b(\proj(A))$, so it is necessarily bounded. Therefore,
$(2)\Leftrightarrow (3)$ is part of the bijections in \cite[Theorem 6.1]{KY}, and it is further equivalent to stating that $(\mathcal{X},\mathcal{Y})=(T'^{\perp_{> 0}},T'^{\perp_{\leq 0}})$ for a silting object $T'$ in $K^b(\proj(A))$. The latter, however, is easily seen to be equivalent to $(5)$.
Moreover, $(5)\Rightarrow (1)$ is trivial, and $(5)\Rightarrow (4)$ holds by \cite[Proposition 2]{NSZ}. Finally, $(4)\Rightarrow(3)$ holds by Theorem \ref{thm lifting}, since $\mathcal{H}$ is the subcategory of finitely presented objects of $\mathbf{H}$. 

Note that if the equivalent conditions in the proposition are fulfilled, then $\Lambda$ and $\Lambda'$ are Morita equivalent,  and $T$ is equivalent to $T'$ as a silting object, that is the smallest subcategory of $D(A)$ containing $T$ and closed under coproducts and direct summands  coincides with the corresponding subcategory containing $T'$.
\end{proof}

Note that Proposition \ref{silting} cannot be easily generalised. Using \cite[Section 4]{BHPST}, there are noetherian rings $A$ admitting pure projective tilting modules that are not equivalent to finitely presented ones. These tilting modules induce compactly generated HRS-t-structures in $D(A)$, which necessarily restrict to $D^b(A)$ by Proposition \ref{HRS}, and which are both silting and cosilting. However, they are not silting with respect to a compact silting object and their hearts are Grothendieck categories with a projective generator, but not module categories.

\smallskip

By the discussion above, the process of lifting t-structures via homotopy colimits is closely related to cosilting theory. In the final part of this section, we will explore a dual way of lifting t-structures, which is closely connected to silting theory. For this purpose, we restrict ourselves again to the setting of finite dimensional algebras. 

\begin{theorem}\label{thm silt}
Let $A$ be a finite dimensional algebra over a field $k$, and 
let $(\mathcal{X},\mathcal{Y})$ be an intermediate t-structure in $D^b(A)$ with heart $\mathcal{H}$. Then the following holds.
\begin{enumerate}
\item $(\mathbf{X},\mathbf{Y})=({}^\perp \mathcal{Y},({}^\perp \mathcal{Y})^\perp)$ is a silting t-structure in $D(A)$ with respect to a bounded silting complex of projective $A$-modules.
\item An object $H \in \mathbf{H}$ satisfies $\Hom_{\mathbf{H}}(H,\mathcal{H})=0$ if and only if $H=0$.
\end{enumerate}
\end{theorem}

\begin{proof}
$(1)$: Since $A$ is a finite dimensional algebra, dual arguments to those in the proof of Lemma \ref{comp.gen.} imply that ${}^\perp \mathcal{Y}={}^\perp \mathcal{S}$, where $\mathcal{S}\subseteq\mathcal{Y}$ is a set of bounded complexes of finitely generated injective $A$-modules. Without loss of generality, we can assume that $\mathcal{S}$ is closed under negative shifts. Our aim is now to show that $({}^\perp \mathcal{S},({}^\perp \mathcal{S})^\perp)$ is a silting t-structure with respect to a bounded silting complex of projective $A$-modules. By $\nu$ and $\nu^{-1}$ we denote the total derived functors of the Nakayama functor $-\otimes_A\mathfrak{D}(A)$ and its right adjoint $\Hom_A(\mathfrak{D}(A),-)$. Here, $\mathfrak{D}=\Hom_k(-,k)$ denotes the standard duality. Recall that $\nu$ and $\nu^{-1}$ induce quasi-inverse triangle equivalences between $K^b(\proj(A))$ and $K^b(\inj(A))$, where the latter denotes the bounded homotopy category of finitely generated injective $A$-modules. Moreover, we have the following natural isomorphism
$$\mathfrak{D}\Hom_{D(A)}(\nu^{-1}M,N)\cong\Hom_{D(A)}(N,M)$$
for all $M$ in $K^b(\inj(A))$ and $N$ in $D(A)$ (see \cite[Ch.1, Section 4.6]{H}, \cite[Section 2.3]{KY} and \cite{KL}).
In our situation, it follows that the subcategory ${}^\perp \mathcal{S}$ coincides with $(\nu^{-1}\mathcal{S})^\perp$, where all the objects in $\nu^{-1}\mathcal{S}$ are compact in $D(A)$. Moreover, by assumption, we know that there are $m,n\in\mathbb{Z}$ with $D(A)^{\leq m}\subseteq{}^\perp \mathcal{S}\subseteq D(A)^{\leq n}$. Hence, we can invoke \cite[Theorem 3.6]{MV} to conclude that $({}^\perp \mathcal{S},({}^\perp \mathcal{S})^\perp)$ is a silting t-structure of the desired form.

$(2)$: Take $H$ in $\mathbf{H}$ with $\Hom_{\mathbf{H}}(H,\mathcal{H})=0$. First, observe that, by using approximation triangles with respect to $(\mathcal{X},\mathcal{Y})$, every object $Y[1]$ in $\mathcal{Y}[1]$ appears in a triangle
$$\xymatrix{H'\ar[r] & Y[1] \ar[r] & Y' \ar[r] & H'[1]}$$
with $H'$ in $\mathcal{H}$ and $Y'$ in $\mathcal{Y}$. Since $\Hom_{D(A)}(H,\mathcal{Y})=0$ and $\Hom_{\mathbf{H}}(H,\mathcal{H})=0$ by assumption, it follows that also $\Hom_{D(A)}(H,\mathcal{Y}[1])=0$. Hence, $H$ belongs to $\mathbf{X}[1]$. As $H$ also belongs to $\mathbf{Y}[1]$, it must be zero, as desired.
\end{proof}

We obtain the following general observation on restricting t-structures.

\begin{corollary}\label{cor restrict}
Let $A$ be a finite dimensional algebra over a field $k$, and 
let $(\mathcal{X},\mathcal{Y})$ be an intermediate t-structure in $D^b(A)$ with heart $\mathcal{H}$. Then there is a unique t-structure $(\mathbf{X},\mathbf{Y})$ in $D(A)$ that restricts to $(\mathcal{X},\mathcal{Y})$ if and only if $\mathcal{H}$ is equivalent to $\fpmod(\Lambda)$ for a finite dimensional $k$-algebra $\Lambda$.
\end{corollary}

\begin{proof}
First, let $(\mathbf{X},\mathbf{Y})$ be the unique t-structure in $D(A)$ restricting to $(\mathcal{X},\mathcal{Y})$. It follows that $(\mathbf{X},\mathbf{Y})$ arises from $(\mathcal{X},\mathcal{Y})$ via the two different ways of lifting discussed in Section \ref{lifting} and Theorem \ref{thm silt}. Thus, we can use Proposition \ref{silting} to conclude that $\mathcal{H}$ is equivalent to $\fpmod(\Lambda)$ for a finite dimensional $k$-algebra $\Lambda$. Conversely, suppose that $\mathcal{H} \cong \fpmod(\Lambda)$ for a finite dimensional $k$-algebra $\Lambda$. We consider the t-structure $(\mathbf{X},\mathbf{Y})$ in $D(A)$ obtained from the lifting procedure in Section \ref{lifting}. Again by Proposition \ref{silting}, it follows that 
$$(\mathbf{X},\mathbf{Y})=(T^{\perp_{> 0}},T^{\perp_{\leq 0}})$$
for a compact silting object $T$. As in the proof of Theorem \ref{thm silt}, we will make use of the derived Nakayama functor $\nu$ and the natural isomorphism
$$\mathfrak{D}\Hom_{D(A)}(M,N)\cong\Hom_{D(A)}(N,\nu M)$$
for $M$ in $K^b(\proj(A))$ and $N$ in $D(A)$.
It follows that $T^{\perp_{> 0}}={}^{\perp_{\leq 0}}C$ for $C:=\nu T[-1]$. Indeed, for an object $X$ in $D(A)$ we have
$$X\in T^{\perp_{> 0}}\Leftrightarrow \Hom_{D(A)}(T,X[n>0])=0$$
$$\Leftrightarrow \Hom_{D(A)}(X[n>0],\nu T)=0$$
$$\Leftrightarrow \Hom_{D(A)}(X,\nu T[-1][n\leq 0])=0\Leftrightarrow X\in {}^{\perp_{\leq 0}}C.$$
Similarly, we have $T^{\perp_{\leq 0}}={}^{\perp_{> 0}}C$ and, thus, $(\mathbf{X},\mathbf{Y})$ is the cosilting t-structure with respect to the cosilting complex $C$ in $K^b(\inj(A))$. Now let $(\mathbf{X}',\mathbf{Y}')$ be any t-structure in $D(A)$ that restricts to $(\mathcal{X},\mathcal{Y})$. Since $T$ belongs to $\mathcal{X}\subset\mathbf{X}'$ and $C$ lies in $\mathcal{Y}\subset\mathbf{Y}'$, we get $\mathbf{X}\subset\mathbf{X}'$ and $\mathbf{Y}\subset\mathbf{Y}'$ showing that $(\mathbf{X},\mathbf{Y})=(\mathbf{X}',\mathbf{Y}')$, as desired.
\end{proof}

Corollary \ref{cor restrict} and its proof show that the two presented ways of lifting t-structures in $D^b(A)$, in general, provide us with distinct classes of t-structures in $D(A)$. In fact, the heart of every t-structure in $D(A)$ that can be obtained via both lifting procedures is necessarily equivalent to a module category.

\smallskip

Moreover, note that part (2) of Theorem \ref{thm silt} is a weaker dual of the generation property of $\mathcal{H}$ in Theorem \ref{thm lifting}(2), which was essential to describe the cosilting t-structures we obtain from lifting.
In general, we do not know how to classify the silting t-structures appearing in Theorem \ref{thm silt}. One of the reasons is that
cogenerating t-structures from a given set of objects seems to be a far more delicate task than generating t-structures. In particular, we cannot expect to simply replace homotopy colimits by homotopy limits to prove dual results to those before. In the case of HRS-t-structures, however, we can provide a complete classification. To make this precise, we need to introduce further terminology. Recall that a subcategory $\mathbf{C}$ of $\Mod(A)$ is called {\bf definable} if it is closed under products, direct limits and pure submodules. Moreover, a submodule is called {\bf pure} if the induced short exact sequence remains exact after tensoring with any other $A$-module. Note that for a torsion pair $(\mathbf{T},\mathbf{F})$ in $\Mod(A)$, the class $\mathbf{T}$ is definable if and only if it is closed under products and pure submodules, and the class $\mathbf{F}$ is definable if and only if it is closed under direct limits, i.e. the torsion pair is of finite type. We are now ready to state the result, which is dual to Proposition \ref{HRS}.

\begin{proposition}\label{HRS dual}
Let $A$ be a finite dimensional algebra over a field $k$.
By assigning to an intermediate t-structure $(\mathcal{X},\mathcal{Y})$ in $D^b(A)$ $($respectively, to a torsion pair $(\mathcal{T},\mathcal{F})$ in $\fpmod(A))$ the t-structure $({}^\perp \mathcal{Y},({}^\perp \mathcal{Y})^\perp)$ in $D(A)$ $($respectively, the torsion pair $({}^\perp \mathcal{F},({}^\perp \mathcal{F})^\perp)$ in $\Mod(A))$, we get a commutative square of bijections as follows.
$$\xymatrix{{\left\{\begin{array}{c}\text{t-structures $(\mathcal{X},\mathcal{Y})$ in $D^b(A)$ with} \\ \text{$D^b(A)^{\leq -1}\subseteq\mathcal{X}\subseteq D^b(A)^{\leq 0}$} \end{array}\right\}}\ar[rr]^{\hspace{0.3cm}1:1}\ar@<-1ex>[dd]_{\operatorname{cogeneration}} & & {\left\{\begin{array}{c}\text{torsion pairs $(\mathcal{T},\mathcal{F})$ in $\fpmod(A)$} \end{array}\right\}}\ar[ll]\ar@<-1ex>[dd]_{\operatorname{cogeneration}} \\  \\{\left\{\begin{array}{c}\text{silting t-structures $(\mathbf{X},\mathbf{Y})$ in $D(A)$} \\ \text{with $D(A)^{\leq -1}\subseteq\mathbf{X}\subseteq D(A)^{\leq 0}$}
\end{array}\right\}}\ar[rr]^{\hspace{0.4cm}1:1}\ar@<-1ex>[uu]_{\operatorname{restriction}} & & {\left\{\begin{array}{c}\text{torsion pairs $(\mathbf{T},\mathbf{F})$ in $\Mod(A)$}\\ \text{with $\mathbf{T}$ definable}
\end{array}\right\}}\ar[ll]\ar@<-1ex>[uu]_{\operatorname{restriction}} }$$
\end{proposition}

\begin{proof}
The upper horizontal bijection was already discussed in Section \ref{section HRS}. Moreover, the lower horizontal bijection is a consequence of the fact that, in our setup, a torsion class $\mathbf{T}$ in $\Mod(A)$ is definable if and only if it is a silting class with respect to a silting $A$-module, which in turn is equivalent to stating that the associated HRS-t-structure is silting (see \cite[Corollary 3.8]{AH} and \cite[Section 4.2]{AMV1}). 
The right vertical map assigning $({}^\perp \mathcal{F},({}^\perp \mathcal{F})^\perp)$ to $(\mathcal{T},\mathcal{F})$ is well-defined by \cite[Section 2.3, Example 2]{CB2}. Since in our setup all torsion pairs in $\Mod(A)$ restrict to $\fpmod(A)$, restriction yields a left inverse. Let us show that these two maps are bijections. 
It suffices to check that a torsion pair $(\mathbf{T},\mathbf{F})$ in $\Mod(A)$ with $\mathbf{T}$ definable is of the form $({}^\perp \mathcal{F},({}^\perp \mathcal{F})^\perp)$, where
$(\mathcal{T},\mathcal{F}):= (\mathbf{T}\cap\fpmod(A),\mathbf{F}\cap\fpmod(A))$ denotes the restricted torsion pair in $\fpmod(A)$. By construction, we know that $({}^\perp \mathcal{F})^\perp\subseteq\mathbf{F}$. It remains to check that ${}^\perp \mathcal{F}\subseteq\mathbf{T}$ as well. Let $X$ be an object in ${}^\perp \mathcal{F}$. By \cite[Section 2.2, Example 3]{CB2}, we can realize $X$ as a pure subobject of $\prod X_i$, where $X_i$ runs through all finite dimensional quotients of $X$. By assumption, all the $X_i$ belong to $\mathcal{T}$. As $\mathcal{T}\subset\mathbf{T}$ and $\mathbf{T}$ is definable, both $\prod X_i$ and $X$ belong to $\mathbf{T}$ as well. 
Finally, it is easy to see that the left vertical map in the diagram above, given by Theorem \ref{thm silt}, makes the square commutative, hence it is a bijection, and its inverse is clearly given by restriction.
This finishes the proof.
\end{proof}


\begin{thebibliography}{99}
\bibitem{AI} T. Aihara, O. Iyama, {\em Silting mutation in triangulated categories}, J. Lond. Math. Soc. 85 (2012), 633--668.
\bibitem{AJSa}  L. Alonso Tarr\'io, A. Jerem\'ias Lop\'ez, M. Saor\'in, {\em Compactly generated t-structures on the derived category of a Noetherian ring}, J. Algebra 324 (2010), 313--346.
\bibitem{AJS}  L. Alonso Tarr\'io, A. Jerem\'ias Lop\'ez, M. Souto Salorio, {\em Constructions of t-structures and equivalences of derived categories}, Trans. Amer. Math. Soc. 355 (2003), 2523--2543.
\bibitem{A} L. Angeleri H\"ugel, {\em On the abundance of silting modules}, Contemp. Math. 716,
Amer. Math. Soc. (2018), 1--23.
\bibitem{AH} L. Angeleri H\"ugel, M. Hrbek, {\em Silting modules over commutative rings}, Int. Math. Res. Not. IMRN 2017, 4131--4151.
\bibitem{AMV1} L. Angeleri H\"ugel, F. Marks, J. Vit\'oria, {\em Silting modules}, Int. Math. Res. Not. IMRN 2016, 1251--1284.
\bibitem{AMV3} L. Angeleri H\"ugel, F. Marks, J. Vit\'oria, {\em Torsion pairs in silting theory}, Pacific J. Math. 291 (2017), 257--278.
\bibitem{BHPST} S. Bazzoni, I. Herzog, P. P\v{r}\'ihoda, J. \v{S}aroch, J. Trlifaj, {\em Pure projective tilting modules}, Doc. Math. 25 (2020), 401--424.
\bibitem{BBD} A. Beilinson, J. Bernstein, P. Deligne, {\em Analyse et topologie sur les espaces singuliers Vol. I}, Ast{\'e}risque 100, Soc. Math. France, Paris, 1982.
\bibitem{BR} A. Beligiannis, I. Reiten, {\em Homological and homotopical aspects of torsion theories}, Mem. Amer. Math. Soc. 188 (2007).
\bibitem{BP} D. Bravo, C. Parra, {\em tCG torsion pairs}, J. Algebra Appl. 18 (2019), 1950127.
\bibitem{CH} L.W. Christensen, T. Holm, {\em The direct limit closure of perfect complexes}, J. Pure Appl. Algebra 219 (2015), 449--463.
\bibitem{CB} W. Crawley-Boevey, {\em Locally finitely presented additive categories}, Commun. Algebra 22 (1994), 1641--1674.
\bibitem{CB2} W. Crawley-Boevey, {\em Infinite-dimensional modules in the representation theory of finite-dimensional algebras}, Algebras and modules, 29--54, CMS Conf. Proc. 23, Amer. Math. Soc., Providence, RI, 1998.
\bibitem{GR} J.R. Garc\'ia Rozas, {\em Covers and Envelopes in the Category of Complexes of Modules}, Chapman \& Hall/CRC Research Notes in Mathematics, 407. Chapman \& Hall/CRC, Boca Raton, FL, 1999.
\bibitem{Groth} M. Groth, {\em Derivators, pointed derivators and stable derivators}, Algebr. Geom. Topol. 13 (2013), 313--374.
\bibitem{H} D. Happel, {\em Triangulated categories in the representation theory of finite dimensional algebras}, London Mathematical Society Lecture Note Series 119, Cambridge University Press, Cambridge, 1988. 
\bibitem{HRS} D. Happel, I. Reiten, S. Smal\o, {\em Tilting in abelian categories and quasitilted algebras}, Mem. Amer. Math. Soc. 120 (1996).
\bibitem{Hr} M. Hrbek, {\em Compactly generated t-structures in the derived category of a commutative ring}, Math. Z. 295 (2020), 47--72.
\bibitem{J} P. J{\o}rgensen, {\em Auslander-Reiten triangles in subcategories}, J. K-Theory 3 (2009), 583--601.
\bibitem{K} B. Keller, {\em Deriving DG categories}, Annales scientifiques de l’É.N.S. 27 (1994), 63--102.
\bibitem{KV} B. Keller, D. Vossieck, {\em Aisles in derived categories}, Bull. Soc. Math. Belg. Ser. A 40 (1988), 239--253.
\bibitem{KY} S. Koenig, D. Yang, {\em Silting objects, simple minded collections, t-structures and co-t-structures for finite dimensional algebras}, Doc. Math. 19 (2014), 403--438.
\bibitem{KL} H. Krause, J. Le, {\em The Auslander-Reiten formula for complexes of modules}, Adv. Math. 207 (2006), 133--148. 
\bibitem{KS} H. Krause, M. Saor\'in, {\em On minimal approximations of modules}, Trends in the representation theory of finite dimensional
algebras, 227--236, Contemp. Math. 229, Amer. Math. Soc., Providence, RI, 1998.
\bibitem{L} R. Laking, {\em Purity in compactly generated derivators and t-structures with Grothendieck hearts}, Math. Z. 295 (2020), 1615--1641.
\bibitem{MV} F. Marks, J. Vit\'oria, {\em Silting and cosilting classes in derived categories}, J. Algebra 501 (2018), 526--544.
\bibitem{NSZ} P. Nicol\'as, M. Saor\'in, A. Zvonareva, {\em Silting theory in triangulated categories with coproducts}, J. Pure Appl. Algebra 223 (2019), 2273--2319.
\bibitem{PSV} C. Parra, M. Saor\'in, S. Virili, {\em Tilting preenvelopes and cotilting precovers in general Abelian categories}, Algebras Represent. Theory (2022), https://doi.org/10.1007/s10468-022-10126-5.
\bibitem {Po} A. Polishchuk, {\em Constant families of t-structures on derived categories of coherent sheaves}, Mosc. Math. J. 7 (2007), 109--134.
\bibitem{PV} C. Psaroudakis, J. Vit\'oria, {\em Realisation functors in tilting theory}, Math. Z. 288 (2018), 965--1028. 
\bibitem{R} J. Rickard, {\em Morita Theory for Derived Categories}, London Math. Soc. s2-39 (1989), 436--456.
\bibitem{RR} J. Rickard, R. Rouquier, {\em Stable categories and reconstruction}, J. Alg. 475 (2017), 287--307.
\bibitem{Sa} M. Saor\'in, {\em On locally coherent hearts}, Pacific J. Math. 287 (2017), 199--221.
\bibitem{SS} M. Saor\'in, J. \v{S}\v{t}ov\'\i\v{c}ek, {\em On exact categories and applications to triangulated adjoints and model structures}, Adv. Math. 228 (2011), 968--1007.
\bibitem{SSt} M. Saor\'in, J. \v{S}\v{t}ov\'\i\v{c}ek, {\em t-structures with Grothendieck hearts via functor categories}, preprint (2020), arXiv:2003.01401.
\bibitem{SSV} M. Saor\'in, J. \v{S}\v{t}ov\'\i\v{c}ek, S. Virili, {\em t-structures on stable derivators and Grothendieck hearts}, preprint (2017), arXiv:1708.07540.
\bibitem{ST} M.J. Souto Salorio, S. Trepode, {\em t-structures on the bounded derived category of the Kronecker algebra}, Appl. Categor. Struct. 20 (2012), 513--529.
\bibitem{S} J. \v{S}\v{t}ov\'\i\v{c}ek, {\em Derived equivalences induced by big cotilting modules}, Adv. Math. 263 (2014), 45--87.

\end{thebibliography}
\end{document}